\documentclass[11pt,reqno]{amsart}

\usepackage{amssymb, amsmath, amsthm}
\usepackage{hyperref}
\usepackage[alphabetic,lite]{amsrefs}
\usepackage{verbatim}
\usepackage{amscd}   
\usepackage[all]{xy} 
\usepackage{youngtab} 
\usepackage{young} 
\usepackage{ytableau}
\usepackage{tikz}
\usepackage{ mathrsfs }
\usepackage{cases}
\usepackage{array}
\usepackage{cellspace}
\usepackage{tabu}
\usepackage{calligra,mathrsfs}
\usepackage{bm}

\newcommand{\defi}[1]{{\upshape\sffamily #1}}
\DeclareMathOperator{\ShHom}{\mathscr{H}\text{\kern -3pt {\calligra\large om}}\,}

\renewcommand{\a}{\alpha}
\renewcommand{\b}{\beta}
\newcommand{\bw}{\bigwedge}

\newcommand{\D}{\mathcal{D}}

\renewcommand{\ll}{\lambda}

\newcommand{\onto}{\twoheadrightarrow}
\newcommand{\oo}{\otimes}

\renewcommand{\P}{\mathcal{P}}

\newcommand{\XA}{\mathscr{X}}

\newcommand{\ZA}{\mathscr{Z}}

\newcommand{\GL}{\operatorname{GL}}

\newcommand{\rk}{\operatorname{rank}}

\newcommand{\Spec}{\operatorname{Spec}}
\newcommand{\Sym}{\operatorname{Sym}}

\renewcommand{\det}{\operatorname{det}}
\newcommand{\dom}{\operatorname{dom}}
\newcommand{\gr}{\operatorname{gr}}

\newcommand{\bb}[1]{\mathbb{#1}}

\newcommand{\mc}[1]{\mathcal{#1}}
\newcommand{\mf}[1]{\mathfrak{#1}}
\newcommand{\ol}[1]{\overline{#1}}
\newcommand{\op}[1]{\operatorname{#1}}

\newcommand{\ul}[1]{\underline{#1}}

\def\PP{{\textbf P}}
\def\lra{\longrightarrow}

\def\hra{\hookrightarrow}

\newtheorem{theorem}{Theorem}[section]
\newtheorem*{theorem*}{Theorem}
\newtheorem*{problem*}{Problem}
\newtheorem{lemma}[theorem]{Lemma}

\newtheorem*{corollary*}{Corollary}

\newtheorem*{main-thm*}{Main Theorem}
\newtheorem*{linear-resolutions*}{Theorem on Linear Resolutions}
\newtheorem*{regularity-powers*}{Theorem on Regularity}
\newtheorem*{injectivity-Ext*}{Theorem on Injectivity of Maps of Ext Modules}
\newtheorem*{Kodaira*}{Kodaira Vanishing for Determinantal Thickenings}

\theoremstyle{definition}

\newtheorem*{definition*}{Definition}

\theoremstyle{remark}
\newtheorem{remark}[theorem]{Remark}
\newtheorem*{remark*}{Remark}

\numberwithin{equation}{section}

\begin{document}

\title{Hodge ideals for the determinant hypersurface}


\author{Michael Perlman}
\address{Department of Mathematics, University of Notre Dame, 255 Hurley, Notre Dame, IN 46556}
\email{mperlman@nd.edu}

\author{Claudiu Raicu}
\address{Department of Mathematics, University of Notre Dame, 255 Hurley, Notre Dame, IN 46556\newline
\indent Institute of Mathematics ``Simion Stoilow'' of the Romanian Academy}
\email{craicu@nd.edu}

\subjclass[2010]{Primary 14M12, 14J17, 14E15, 13D45}

\date{\today}

\keywords{Hodge ideals, determinantal ideals, Hodge filtration, weight filtration, local cohomology}

\begin{abstract} 
 We determine explicitly the Hodge ideals for the determinant hypersurface as an intersection of symbolic powers of determinantal ideals. We prove our results by studying the Hodge and weight filtrations on the mixed Hodge module $\mc{O}_{\XA}(*\ZA)$ of regular functions on the space $\XA$ of $n\times n$ matrices, with poles along the divisor $\ZA$ of singular matrices. The composition factors for the weight filtration on $\mc{O}_{\XA}(*\ZA)$ are pure Hodge modules with underlying $\D$-modules given by the simple $\GL$-equivariant $\D$-modules on $\XA$, where $\GL$ is the natural group of symmetries, acting by row and column operations on the matrix entries. By taking advantage of the $\GL$-equivariance and the Cohen--Macaulay property of their associated graded, we describe explicitly the possible Hodge filtrations on a simple $\GL$-equivariant $\D$-module, which are unique up to a shift determined by the corresponding weights. For non-square matrices, $\mc{O}_{\XA}(*\ZA)$ is replaced by the local cohomology modules $H^{\bullet}_{\ZA}(\XA,\mc{O}_{\XA})$, which turn out to be pure Hodge modules. By working out explicitly the Decomposition Theorem for some natural resolutions of singularities of determinantal varieties, and using the results on square matrices, we determine the weights and the Hodge filtration for these local cohomology modules.

\end{abstract}

\maketitle

\section{Introduction}\label{sec:intro}

To any smooth complex variety $X$ and reduced divisor $Z\subset X$ one can associate the $\D$-module
\begin{equation}\label{eq:OX*Z}
 \mc{O}_X(*Z) = \bigcup_{k\geq 0}\mc{O}_X(kZ),
\end{equation}
consisting of regular functions on $X$ with poles along $Z$. When $X$ is affine with coordinate ring $S$, and $Z$ is defined by the equation $f=0$, the module $\mc{O}_X(*Z)$ is the localization $S_f$. The module $\mc{O}_X(*Z)$ is equipped with the \defi{Hodge filtration} $F_{\bullet}(\mc{O}_X(*Z))$ \cites{saito-MHM,MP-hodge-ideals}, satisfying
\begin{equation}\label{eq:Hodge-in-pole}
 F_{k}(\mc{O}_X(*Z)) \subseteq \mc{O}_X((k+1)Z)\mbox{ for }k\geq 0,
\end{equation}
with equality when $Z$ itself is smooth. In general however, the Hodge filtration is a subtle invariant measuring the singularities of $Z$. Following \cite{MP-hodge-ideals}, we note that the data of the Hodge filtration is equivalent to the sequence of \defi{Hodge ideals} of $Z$, determined by the equality
\begin{equation}\label{eq:Fk-to-Ik}
 F_k(\mc{O}_X(*Z)) = I_k(Z) \oo \mc{O}_X((k+1)Z)\mbox{ for }k\geq 0.
\end{equation}

The goal of this paper is to describe explicitly the Hodge ideals of the determinant hypersurface. We let $\XA=\bb{C}^{n\times n}$, let $S = \bb{C}[x_{i,j}]$ denote the coordinate ring of $\XA$, let $\det=\det(x_{i,j})$ denote the determinant of the generic $n\times n$ matrix, and let $\ZA$ denote the determinant hypersurface consisting of matrices with vanishing determinant. For $1\leq p\leq n$ we let $J_p$ denote the ideal generated by the $p\times p$ minors of $(x_{i,j})$, corresponding to the variety $\ZA_{p-1}\subset \XA$ of matrices of rank $<p$. We write $J_p^{(d)}$ for the \defi{$d$-th symbolic power} of $J_p$, consisting of regular functions that vanish to order $d$ along $\ZA_{p-1}$, with the convention that $J_p^{(d)} = S$ when $d\leq 0$.

\begin{theorem}\label{thm:hodge-ideals-detl}
The Hodge ideals of $\ZA$ are given by
 \begin{equation}\label{eq:Ik=intersec-symbolic}
  I_k(\ZA) = \bigcap_{p=1}^{n-1} J_p^{\left((n-p)\cdot(k-1) - {n-p\choose 2}\right)}\mbox{ for }k\geq 0.
 \end{equation}
\end{theorem}

It follows from (\ref{eq:Ik=intersec-symbolic}) that $I_0(\ZA)=I_1(\ZA)=S$ and $I_2(\ZA) = J_{n-1}$, which was established at the set-theoretic level in \cite[Example~20.14]{MP-hodge-ideals}. Since $\ZA$ has multiplicity $m=n-p+1$ along $\ZA_{p-1}=V(J_p)$, and $\ZA_{p-1}$ has codimension $r=(n-p+1)^2$ in $\XA$, it follows from \cite[Theorem~E]{MP-hodge-ideals} that $I_k(\ZA)\subseteq J_p^{(q)}$ for $q=\min\{n-p,(n-p+1)\cdot(k-n+p)\}$, whereas (\ref{eq:Ik=intersec-symbolic}) implies that when $k$ is large, the optimal value of $q$ is given by $(n-p)\cdot(k-1) - {n-p\choose 2}$. 

We prove our results by taking advantage of the rich symmetry coming from the action of the group $\GL=\GL_n(\bb{C})\times\GL_n(\bb{C})$ on $\XA$ (via row and column operations), which preserves $\ZA$, along with all the determinantal varieties $\ZA_{p-1}$, $p\leq n$. It follows that the Hodge ideals $I_k(\ZA)$ and the filtered pieces $F_k(\mc{O}_{\XA}(*\ZA))$ are $\GL$-subrepresentations of $\mc{O}_{\XA}(*\ZA)$. Every such subrepresentation $M$ can be described in terms of its irreducible decomposition, which in turn is completely determined by a subset $\mf{W}(M)$ of the set of \defi{dominant weights} $\bb{Z}^n_{\dom}$ (see Sections~\ref{subsec:GL-reps} and~\ref{subsec:equiv-Dmods} for more details). We prove the following.

\begin{theorem}\label{thm:hodge-filtration-detl}
For $k\in\bb{Z}$ we let
\[ U_k^p = \left\{ \ll\in\bb{Z}^n_{\dom} : \ll_p\geq p-n\geq \ll_{p+1},\ \ll_{p+1}+\cdots+\ll_n \geq -{n-p+1\choose 2} - k \right\},\mbox{ for }0\leq p\leq n.\]
The generation level for the Hodge filtration on $\mc{O}_{\XA}(*\ZA)$ is ${n\choose 2}$. Moreover, we have
\begin{equation}\label{eq:hodge-filtration-detl}
  \mf{W}(F_k(\mc{O}_{\XA}(*\ZA))) = \bigsqcup_{p=0}^n U_k^p.
\end{equation}
\end{theorem}

It is interesting to compare the assertion about the generation level (see Section \ref{subsec:Hodge-detl}) in Theorem~\ref{thm:hodge-filtration-detl} with \cite[Theorem~A]{MP-min-exp}, which for $n\geq 2$ gives the upper bound $\dim(\ZA)-\tilde{\a}_{\ZA}$ for the generation level, where $\tilde{\a}_{\ZA}$ is the minimal exponent of the singular divisor $\ZA$. Since the reduced Bernstein--Sato polynomial of $\ZA$ is $(s+2)\cdots(s+n)$, we have that $\tilde{\a}_{\ZA}=2$ and therefore $\dim(\ZA)-\tilde{\a}_{\ZA}=n^2-3$. For $n=2$ this agrees with the level ${n\choose 2}$ that we determine, which is also a consequence of \cite[Theorem~0.7]{saito-gen-level} since $\ZA$ is a homogeneous isolated singularity (the affine cone over $\PP^1\times\PP^1$). For $n\geq 3$ however, there is a strict inequality $n^2-3>{n\choose 2}$.

The equivalence between (\ref{eq:Ik=intersec-symbolic}) and (\ref{eq:hodge-filtration-detl}) is established in Section~\ref{subsec:equiv-ideals-filtration}. To prove Theorem~\ref{thm:hodge-filtration-detl}, we analyze the structure of $\mc{O}_{\XA}(*\ZA)$ as a mixed Hodge module. For each $p=0,\cdots,n$ we let $D_p = \mc{L}(\ZA_p,\XA)$ denote the intersection homology module associated to $\ZA_p$, and let $IC_{\ZA_p}^H$ denote the Hodge module on $\XA$ corresponding to the trivial variation of Hodge structure on the orbit $O_p$ (see Section~\ref{subsec:Hodge-IC}). Up to a Tate twist, $IC_{\ZA_p}^H$ is the only Hodge module with underlying $\D$-module $D_p$. We write $W_{\bullet}$ for the weight filtration, and $\gr_{\bullet}^W$ for the associated graded with respect to $W_{\bullet}$, and prove the following.

\begin{theorem}\label{thm:weight-filtration}
 We have that $\gr^W_w\mc{O}_{\XA}(*\ZA)=0$ if $w<n^2$ or $w>n^2+n$, and
 \[ \gr^W_{n^2+n-p}\mc{O}_{\XA}(*\ZA) = IC_{\ZA_p}^H\left(-{n-p+1\choose 2}\right)\mbox{ for }p=0,\cdots,n.\]
\end{theorem}

To go from Theorem~\ref{thm:weight-filtration} to (\ref{eq:hodge-filtration-detl}), we need to understand the Hodge filtration on each $IC_{\ZA_p}^H$ (since Tate twists only amount to a shift in $F_{\bullet}$). We do so in Theorem~\ref{thm:Hodge-Dp}, in the more general case when $\XA$ is a space of rectangular (not necessarily square) matrices. The following is a consequence of Theorem~\ref{thm:Hodge-Dp}.

\begin{theorem}\label{thm:Hodge-ICZp}
 If $\XA\simeq\bb{C}^{n\times n}$ and $F_{\bullet}$ is the Hodge filtration on the $\D$-module $D_p$ underlying $IC_{\ZA_p}^H$, then 
 \[\mf{W}(F_k(D_p)) = U_{k-{n-p+1\choose 2}}^p\mbox{ for all }k\in\bb{Z},\]
 and in particular $F_k(D_p)$ is non-zero if and only if $k\geq(n-p)^2$.
\end{theorem}

We reformulate the last assertion in Theorem~\ref{thm:Hodge-ICZp} by saying that the Hodge filtration for $IC_{\ZA_p}^H$ starts in level $(n-p)^2$ (the codimension of $\ZA_p$ in $\XA$), which is in fact also the generation level. Combining this with Theorem~\ref{thm:weight-filtration}, we get that the Hodge filtration for $\gr^W_{n^2+n-p}\mc{O}_{\XA}(*\ZA)$ starts (and is generated) in level ${n-p\choose 2}$, which is maximized for $p=0$. This explains the assertion about the generation level in Theorem~\ref{thm:hodge-filtration-detl}. The special case $p=n$ in Theorem~\ref{thm:Hodge-ICZp} is easy to understand: we have $D_n = S$, and for $k\geq 0$ we have that $U_k^n = \{\ll\in\bb{Z}^n_{\dom}:\ll_n\geq 0\}$ is the set of all partitions with at most $n$ parts (independently on $k$); this reflects the fact that $S$ has the trivial Hodge filtration $F_k(S)=S$ for all $k\geq 0$, and that $\mf{W}(S)=U_k^n$ is determined by Cauchy's formula (\ref{eq:decomp-S}).

If we consider instead non-square matrices $\XA\simeq\bb{C}^{m\times n}$, $m>n$, then the variety $\ZA$ of singular matrices is no longer a divisor. Nevertheless, the local cohomology groups $H^{\bullet}_{\ZA}(\XA,\mc{O}_{\XA})$ replace $\mc{O}_{\XA}(*\ZA)$ and have a natural structure of (mixed) Hodge modules. We know from \cite[(5.1)]{raicu-dmods} and \cites{raicu-weyman,raicu-weyman-witt} that the only non-zero local cohomology groups are
\begin{equation}\label{eq:Dp-as-loccoh}
 D_p = H^{1+(n-p)\cdot(m-n)}_{\ZA}(\XA,\mc{O}_{\XA})\mbox{ for }p=0,\cdots,n-1,
\end{equation}
where $D_p=\mc{L}(\ZA_p,\XA)$ as before. By Theorem~\ref{thm:Hodge-Dp}, the Hodge filtration is determined by the weights of the corresponding Hodge modules, which are given as follows.

\begin{theorem}\label{thm:hodge-filtration-loccoh}
 For each $p=0,\cdots,n-1$, the local cohomology group $H^{1+(n-p)\cdot(m-n)}_{\ZA}(\XA,\mc{O}_{\XA})$ is a pure Hodge module of weight $mn+(n-p)\cdot(m-n+1)$. 
\end{theorem}

To explain the proof strategy for Theorem~\ref{thm:hodge-filtration-loccoh}, and the implicit choice of Hodge structure on local cohomology, we introduce some notation: given a smooth variety $X$ we write $\mc{O}_X^H=IC_X^H$ for the trivial Hodge module on $X$; for a morphism $f$ between smooth varieties we write $f_+$ for the direct image functor on the derived category of mixed Hodge modules (and use the same notation for the corresponding $\D$-module direct image functor). We let $U=O_n$ denote the dense orbit of nonsingular matrices, and write $f:U\lra\XA$ for the inclusion map. When $\XA\simeq\bb{C}^{n\times n}$ we have $f_+\mc{O}_U^H = \mc{O}_{\XA}(*\ZA)$, which gives the mixed Hodge module structure that was implicit in our earlier discussion. When $\XA\simeq\bb{C}^{m\times n}$, $m>n$, we have $\mc{H}^0(f_+\mc{O}_U^H) = \mc{O}_{\XA}^H$, and 
\begin{equation}\label{eq:Rf+=loccoh}
 \mc{H}^j(f_+\mc{O}_U^H) = H^{1+j}_{\ZA}(\XA,\mc{O}_{\XA})\mbox{ for }j>0.
\end{equation}
To understand $f_+\mc{O}_U^H$, we factor $f$ as a composition
\begin{equation}\label{eq:factor-f}
 U \overset{\iota}{\lra} Y \overset{\pi}{\lra} \XA
\end{equation}
where $\iota$ is an affine open immersion, $\pi$ is projective birational, and $Y$ is locally identified with a space of $n\times n$ matrices over an $n\cdot(m-n)$-dimensional base. More precisely, we consider the Grassmannian $\bb{G}=\bb{G}(n;m)$ with tautological rank $n$ bundle $\mc{Q}$, and let $Y = \bb{A}_{\bb{G}}(\mc{Q} \oo \bb{C}^n)$ denote the corresponding geometric vector bundle. Writing $\ZA^Y=Y\setminus U$ we have $\iota_+\mc{O}_U^H=\mc{O}_Y(*\ZA^Y)$, which we understand using the case of square matrices: we have a rank stratification on $Y$ by subvarieties $\ZA^Y_p$, and the composition factors for the weight filtration on $\mc{O}_Y(*\ZA^Y)$ are given by $D^Y_p=\mc{L}(\ZA^Y_p,Y)$ (with the appropriate Hodge structure). The conclusion now follows from a spectral sequence argument combined with the following explicit consequence of the Decomposition Theorem. For $a\geq b$ we consider the $q$-binomial coefficients
\[{a\choose b}_q = \frac{(1-q^a)\cdot(1-q^{a-1})\cdots (1-q^{a-b+1})}{(1-q^b)\cdot(1-q^{b-1})\cdots (1-q)},\]
and make the convention that ${a\choose b}_q=0$ if $a<b$.

\begin{theorem}\label{thm:decomp-DpY}
 For each $0\leq p\leq n$, the $\D$-module direct image of $D_p^Y$ is given by the formal identity
 \[\sum_{j\in\bb{Z}} \mc{H}^j(\pi_{+}D_p^Y)\cdot q^j = q^{-(n-p)\cdot(m-n)}\cdot{m-p\choose n-p}_{q^2}\cdot \sum_{i=0}^p D_i \cdot q^{-(m-n-p+i)\cdot(p-i)}\cdot{m-n\choose p-i}_{q^2}.\]
\end{theorem}

Implicit in the above formula is the fact that $\mc{H}^j(\pi_{+}D_p^Y)$ is semisimple, and its decomposition as a direct sum of copies of the modules $D_0,\cdots,D_p$ is obtained by equating the coefficients in the formal identity. As a sanity check, we consider the case when $p=n$ and $m=n+1$, when we have that $D_p^Y = \mc{O}_Y$ and $\pi$ is a semismall map with relevant strata $O_n$ and $O_{n-1}$ \cite[Lecture~3]{deCat-notes}, \cite{deCat-Mig}. The $q$-binomials ${m-n\choose p-i}_{q^2}$ are non-zero only for $i=n$ and $i=n-1$, and the formula in Theorem~\ref{thm:decomp-DpY} becomes
\[ \pi_{+}\mc{O}_Y = D_n \oplus D_{n-1}.\]
Specializing further to the case $n=1$, we get that $\XA=\bb{A}^2$ and $Y$ is the blow-up of $\XA$ at the origin, $D_1=\mc{O}_{\XA}$ and $D_0$ is the simple $\D$-module supported at the origin, a familiar example of the Decomposition Theorem.

\medskip

\noindent{\bf Organization.} In Section~\ref{sec:prelim} we recall basic notions from representation and $\D$-module theory, and some properties of spaces of matrices. In Section~\ref{sec:Hodge-Dp} we characterize the possible Hodge filtrations for a simple equivariant $\D$-module on $m\times n$ matrices. In Section~\ref{sec:hodge-det} we determine the weight and Hodge filtrations on the localization $S_{\det}$ at the determinant, and deduce the description of the Hodge ideals for the determinant hypersurface. We end with a discussion of the Hodge structure on local cohomology in Section~\ref{sec:hodge-loccoh}. 
\section{Preliminaries}\label{sec:prelim}

In this section we establish some notation and review basic facts that will be needed in the paper, regarding spaces of matrices, affine bundles, Grassmannians and flag varieties, representations of the general linear group, equivariant $\D$-modules, and the Hodge filtration on an intersection cohomology $\D$-module. We work throughout with varieties of finite type over $\bb{C}$. For any such variety $X$, we let $d_X$ denote its dimension. All our $\D$-modules are left $\D$-modules. Tensor products are considered over $\bb{C}$ unless otherwise stated. 

\subsection{Spaces of matrices, conormal varieties}\label{subsec:matrices}

Consider positive integers $m\geq n$ and complex vector spaces $V_1,V_2$, $\dim(V_1)=m$, $\dim(V_2)=n$. We write $S=\Sym(V_1\oo V_2)$ for the symmetric algebra of $V_1\oo V_2$, and let $\XA = \op{Spec}(S)$ denote the corresponding affine space, whose $\bb{C}$-points are parametrized by $V_1^{\vee}\oo V_2^{\vee}$, where $V^{\vee}$ denotes the dual of a vector space $V$. A choice of bases for $V_1,V_2$ induces identifications $S \simeq \bb{C}[x_{i,j}]$ and $\XA\simeq\bb{C}^{m\times n}$ (the space of $m\times n$ matrices). We write $\GL(V)$ for the group of invertible linear transformations of a vector space $V$, and let $\GL=\GL(V_1)\times\GL(V_2)$. There is a natural $\GL$-action on $\XA$, with orbits $O_p$ consisting of matrices of rank $p$, $p=0,\cdots,n$. We write $\ZA_p=\ol{O}_p$ for the corresponding orbit closures. If we let $J_p\subseteq S$ denote the ideal generated by the $p\times p$ minors of the matrix $(x_{i,j})$ (which does not depend on the choice of bases in $V_1,V_2$), then the defining ideal of $\ZA_p$ is $J_{p+1}$. We write $d_p$ (resp. $c_p$) for the dimension (resp. codimension) of $\ZA_p$ (in $\XA$), which are computed by
\begin{equation}\label{eq:dp-cp}
 d_p = p\cdot(m+n-p)\quad\mbox{ and }\quad c_p=(m-p)\cdot(n-p).
\end{equation}

We let $S'=\Sym(V_1^{\vee}\oo V_2^{\vee})$ and $\XA'=\op{Spec}(S')$, and define $O'_p$, $\ZA'_p$, $J'_p$ in analogy to the previous paragraph. A choice of basis for $V_1,V_2$ determines dual bases on $V_1^{\vee},V_2^{\vee}$, and an identification $S'\simeq\bb{C}[y_{i,j}]$. The cotangent space $T^*\XA$ is naturally identified with $\XA\times\XA' = \op{Spec}(A)$, where $A=S\oo S'$ . We write $\pi,\pi'$ for the projections from $T^*\XA$ to the two factors. We write $C_p$ for the \defi{conormal variety of $\ZA_p$}, which is the closure in $T^*\XA$ of the conormal bundle to $O_p$. As a set, it consists of (see \cite{strickland})
\begin{equation}\label{eq:explicit-Cp}
 C_p = \{(x,x')\in \ZA_p \times \ZA'_{n-p} : xx'=0,\ x'x=0\},
\end{equation}
where $xx'$ and $x'x$ denote matrix multiplications, or in more invariant terms, are defined by the contraction maps from $V_1^{\vee}\oo V_2^{\vee} \oo V_1 \oo V_2$ to $V_1^{\vee} \oo V_1$ and $V_2^{\vee} \oo V_2$, induced by the natural pairings $V_i^{\vee}\oo V_i\to\bb{C}$.

It follows from (\ref{eq:explicit-Cp}) that $\pi(C_p)=\ZA_p$ and $\pi'(C_p)=\ZA'_{n-p}$. Therefore, if we let $I(C_p)\subseteq A$ denote the defining ideal of $C_p$, and if we think of $S,S'$ as subrings of $A$ in the natural way, then
\begin{equation}\label{eq:proj-ICp}
 I(C_p) \cap S = J_{p+1}\quad\mbox{ and }\quad I(C_p) \cap S' = J'_{n-p+1}.
\end{equation}

\subsection{Representations of the general linear group}\label{subsec:GL-reps}

For a vector space $V\simeq\bb{C}^N$ we have $\GL(V)\simeq\GL_N(\bb{C})$ and the irreducible finite dimensional $\GL(V)$-representations are classified by the set of \defi{dominant weights}
\[ \bb{Z}^N_{\dom} = \{ \ll \in\bb{Z}^N : \ll_1 \geq \ll_2 \geq \cdots \geq \ll_N\}.\]  
We write $\bb{S}_{\ll}V$ for the irreducible representation with \defi{highest weight} $\ll\in\bb{Z}^N_{\dom}$, and have for instance
\[ \bb{S}_{\ll}V = \Sym^d V\mbox{ when }\ll=(d,0^{N-1}),\ d\geq 0,\mbox{ and } \bb{S}_{\ll}V = \bw^r V\mbox{ when }\ll=(1^r,0^{N-r}),\ 0\leq r\leq N.\]
Taking duals, we obtain isomorphisms
\begin{equation}\label{eq:lam-dual}
 \bb{S}_{\ll}(V^{\vee}) \simeq (\bb{S}_{\ll}V)^{\vee} \simeq \bb{S}_{\ll^{\vee}}V,\mbox{ where }\ll^{\vee} = (-\ll_N,-\ll_{N-1},\cdots,-\ll_1).
\end{equation}

When $\ll_N\geq 0$ we say that $\ll$ is a \defi{partition}, which we typically write by omitting any trailing zeros. We write $\mc{P}_N=\{\ll\in\bb{Z}^N_{\dom}:\ll_N\geq 0\}$ for the set of partitions with at most $N$ parts, and think of $\mc{P}_N$ as a subset of $\mc{P}_{N+1}$ by setting $\ll_{N+1}=0$ for $\ll\in\mc{P}_N$. With these conventions, we have $\mc{P}_n\subseteq\mc{P}_m$ for $m\geq n$, and if $V_1,V_2$ are as in Section~\ref{subsec:matrices}, then by Cauchy's formula \cite[Corollary 2.3.3]{weyman} we get a decomposition into irreducible $\GL$-representations
\begin{equation}\label{eq:decomp-S}
 S = \Sym(V_1\oo V_2) = \bigoplus_{\ll\in\mc{P}_n} \bb{S}_{\ll}V_1 \oo \bb{S}_{\ll}V_2.
\end{equation}
The component $\bw^p V_1\oo\bw^p V_2$ in (\ref{eq:decomp-S}) occurs for $\ll=(1^p)$ and corresponds to the linear span of the $p\times p$ minors of $(x_{i,j})$, the generators of the ideal $J_p$. Moreover, we have that
\begin{equation}\label{eq:Slam-in-Jp}
 \bb{S}_{\ll}V_1 \oo \bb{S}_{\ll}V_2 \subset J_p \Longleftrightarrow \ll_p\geq 1.
\end{equation}
  
\subsection{Hodge filtration on an IC module}\label{subsec:Hodge-IC}

In this section $X$ is a smooth variety and $Z\subseteq X$ is an irreducible closed subvariety. We write $\mc{L}(Z,X)$ for the \defi{intersection cohomology (simple) $\D$-module} corresponding to the trivial local system on the regular part $Z_{reg}$ of $Z$ \cite[Remark~7.2.10]{hottaetal}. We write $IC_Z^H$ for the Hodge module on $X$ corresponding to the trivial variation of Hodge structure on $Z_{reg}$, so that $IC_Z^H$ is pure of weight $d_Z$ \cite[Section~8.3.3]{hottaetal}. When $Z=X$ we have $\mc{L}(X,X)=\mc{O}_X$ and we write $\mc{O}_X^H$ instead of $IC_X^H$. Every Hodge module on $X$ with underlying $\D$-module $\mc{L}(Z,X)$ is obtained by applying a Tate twist to the trivial variation of Hodge structure: we write $IC_Z^H(k)$ for the resulting Hodge module, which is pure of weight $d_Z-2k$.

\begin{lemma}\label{lem:hodgeF-LZX}
 The Hodge filtration $F_{\bullet}$ for $IC_Z^H(k)$ starts in level $d_X-d_Z+k$, that is,
 \[ F_p(\mc{L}(Z,X)) = 0\mbox{ for }p<d_X-d_Z+k,\quad F_{d_X-d_Z+k}(\mc{L}(Z,X))\neq 0.\]
\end{lemma}

\begin{proof}
 Since the Tate twist $(k)$ replaces $F_{\bullet}$ by $F_{\bullet-k}$, it suffices to consider the case $k=0$. Suppose first that $Z=X$. The Hodge filtration for $\mc{O}_X^H$ is given by $F_p(\mc{O}_X)=0$ for $p<0$ and $F_p(\mc{O}_X)=\mc{O}_X$ for $p\geq 0$, so the conclusion follows. Suppose next that $Z$ is smooth, so that $IC_Z^H = i_+\mc{O}_Z^H$, where $i:Z\hra X$ is the inclusion. The conclusion now follows from the description of the filtration on the direct image in \cite[Section~8.3.3]{hottaetal} (we have $F_q(D_{X\leftarrow Z}) \neq 0$ if and only if $q\geq 0$, and $F_{p-q+d_Z-d_X}(\mc{O}_Z)\neq 0$ if and only if $p-q+d_Z-d_X\geq 0$).
 
 Finally, consider the general case when $Z\subseteq X$ is an irreducible subvariety, and let $U\subset X$ be an open subset such that $U\cap Z = Z_{reg}$. By the previous discussion, we have $F_p(\mc{L}(Z_{reg},U))\neq 0$ if and only if $p\geq d_U-d_{Z_{reg}}=d_X-d_Z$. Since $F_p(\mc{L}(Z_{reg},U)) = F_p(\mc{L}(Z,X))_{|_U}$, this implies that $F_p(\mc{L}(Z,X))\neq 0$ for $p\geq d_X-d_Z$. If $F_p(\mc{L}(Z,X))\neq 0$ for some $p<d_X-d_Z$, then $F_p(\mc{L}(Z,X))$ has support contained in the proper closed subset $Z_{sing} = Z\setminus Z_{reg}$ of $Z$, and therefore the local cohomology module $\mc{H}_{Z_{sing}}^0(\mc{L}(Z,X))$ is a proper $\D$-submodule of the simple $\D$-module $\mc{L}(Z,X)$, a contradiction.
\end{proof}

\subsection{$\GL$-equivariant $\D$-modules on $\bb{C}^{m\times n}$}\label{subsec:equiv-Dmods}

We let $\XA\simeq\bb{C}^{m\times n}$ as in Section~\ref{subsec:matrices}, and consider the category $\op{mod}_{\GL}(\D_{\XA})$ of $\GL$-equivariant (holonomic) coherent $\D$-modules on $\XA$. The simple objects in $\op{mod}_{\GL}(\D_{\XA})$ are the $\D$-modules $D_p = \mc{L}(\ZA_p,\XA)$, $p=0,\cdots,n$. Their $\GL$-structure is given by \cite[Section~5]{raicu-dmods}
\begin{equation}\label{eq:char-Dp}
 D_p = \bigoplus_{\ll\in W^p} \bb{S}_{\ll(p)}V_1 \oo \bb{S}_{\ll}V_2,
\end{equation}
where $\ll(p) = (\ll_1,\cdots,\ll_p,(p-n)^{m-n},\ll_{p+1}+(m-n),\cdots,\ll_n+(m-n))$, and
\begin{equation}\label{eq:def-Wp}
 W^p = \{ \ll\in\bb{Z}^n_{\dom} : \ll_p\geq p-n,\ \ll_{p+1}\leq p-m\},\mbox{ for }p=0,\cdots,n.
\end{equation}
We note that in the special case $p=n$ we have $D_n=S$, $W^n=\mc{P}_n$, and (\ref{eq:char-Dp}) reduces to Cauchy's formula.

As explained in \cite[Theorem~5.4]{lor-wal}, the category $\op{mod}_{\GL}(\D_{\XA})$ is semisimple when $m\neq n$, and it is an explicit quiver category for $m=n$. In the case when $m=n$, there exists a unique non-trivial extension of $D_p$ by $D_{p+1}$, which is constructed as follows. We write $\det$ for any non-zero generator of the $1$-dimensional representation $\bw^n V_1 \oo \bw^n V_2 \subset S$. After choosing basis on $V_1,V_2$ as before, $\det$ can be identified with the determinant of the matrix of variables $(x_{i,j})$. The localization $S_{\det}$ is an element of $\op{mod}_{\GL}(\D_{\XA})$, and admits a filtration (see \cite[Theorem~1.1]{raicu-dmods})
 \begin{equation}\label{eq:filtration-Sdet}
 0\subsetneq S\subsetneq\langle\det^{-1}\rangle_{\D}\subsetneq\cdots\subsetneq\langle\det^{-n}\rangle_{\D}=S_{\det},
\end{equation}
with associated composition factors $D_0,\cdots,D_n$, where $D_p \simeq \langle\det^{p-n}\rangle_{\D} / \langle\det^{p-n+1}\rangle_{\D}$ for $0\leq p<n$, $D_n\simeq S$. The non-trivial extension of $D_p$ by $D_{p+1}$ arises as the quotient $\langle\det^{p-n}\rangle_{\D} / \langle\det^{p-n+2}\rangle_{\D}$ for $0\leq p\leq n-2$, and as $\langle\det^{-1}\rangle_{\D}$ for $p=n-1$. The filtration (\ref{eq:filtration-Sdet}) completely describes the lattice of submodules of $S_{\det}$.

We notice also that (in the case $m=n$) we have
\begin{equation}\label{eq:Sdet-decomp}
 S_{\det} = \bigoplus_{\ll\in\bb{Z}^n_{\dom}} \bb{S}_{\ll}\bb{C}^n \oo \bb{S}_{\ll}\bb{C}^n.
\end{equation}
The sets $W^0,\cdots,W^n$ in (\ref{eq:def-Wp}) form a partition of $\bb{Z}^n_{\dom}$, reflecting the fact that as a $\GL$-representation, $S_{\det}$ is isomorphic to the direct sum $D_0\oplus\cdots\oplus D_n$. It follows from (\ref{eq:Sdet-decomp}) that every $\GL$-subrepresentation $M\subseteq S_{\det}$ is uniquely determined by a subset of $\bb{Z}^n_{\dom}$, which we denote $\mf{W}(M)$. Moreover, if $M$ is an $S$-submodule of $S_{\det}$ then we have the implication
\begin{equation}\label{eq:W-module-interval}
 \mbox{if }\ll\in\mf{W}(M)\mbox{ and }\mu\geq\ll\mbox{ then }\mu\in\mf{W}(M).
\end{equation}
This property is not satisfied by $\mf{W}(D_p)=W^p$ unless $p=n$, but it is satisfied by $\mf{W}(D_n\oplus\cdots\oplus D_p) = \{\ll\in\bb{Z}^n_{\dom}: \ll_p\geq p-n\}$, since it describes the underlying $\GL$-representation of $\langle\det^{p-n}\rangle_{\D}$. 

\subsection{Affine bundles, Grassmannians, flag varieties}\label{subsec:bundles}

For a coherent locally free sheaf $\mc{E}$ on a variety $B$, we consider the \defi{geometric affine bundle} associated to $\mc{E}$ to be
\[\bb{A}_B(\mc{E}) = \ul{\Spec}_{\mc{O}_B}\Sym(\mc{E}),\quad\mbox{where}\quad\Sym(\mc{E}) = \mc{O}_B \oplus \mc{E} \oplus \Sym^2\mc{E} \oplus \cdots\]
Any surjection $\mc{E}\onto\mc{F}$ induces a closed immersion $\bb{A}_B(\mc{F})\hookrightarrow\bb{A}_B(\mc{E})$. Our main example of affine bundles is
 \[ \XA_B(\mc{E}_1,\mc{E}_2) = \bb{A}_B(\mc{E}_1\oo\mc{E}_2),\mbox{ where }\rk(\mc{E}_i) = r_i,\]
which is locally isomorphic to a space of $r_1\times r_2$ matrices over the base $B$. It then has a natural rank stratification, and we let $\ZA_{B,p}(\mc{E}_1,\mc{E}_2) \subseteq \XA_B(\mc{E}_1,\mc{E}_2)$ denote the loci of rank $\leq p$ matrices. The special case $B=\op{Spec}(\bb{C})$ and $\mc{E}_i=V_i$ recovers our earlier definition of $\XA$ from Section~\ref{subsec:matrices}.
 
We write $\bb{G}(p;V)$ for the Grassmannian parametrizing $p$-dimensional quotients of a vector space $V$, and write $\bb{G}(p;N)$ for $\bb{G}(p;\bb{C}^N)$. We consider the tautological exact sequence on $\bb{G}(p;V)$
\[0\lra \mc{R}_{N-p} \lra V\oo\mc{O}_{\bb{G}(p;V)} \lra \mc{Q}_p \lra 0\]
where $N=\dim(V)$, $\rk(\mc{R}_{N-p})=N-p$, $\rk(\mc{Q}_p)=p$. We will also consider $2$-step partial flag varieties $\bb{F}(n,p;V)$ for $n>p$, and write $\mc{Q}_n$ and $\mc{Q}_p$ for the corresponding tautological quotient sheaves. We note that $\bb{F}(n,p;V)$ can be interpreted as a relative Grassmannian in two ways: parametrizing rank $p$ quotients of the sheaf $\mc{Q}_n$ on $\bb{G}(n;V)$, in which case we get a $\bb{G}(p,n)$-bundle over $\bb{G}(n;V)$; or, as parametrizing rank $(n-p)$ quotients of the sheaf $\mc{R}_{N-p}$ on $\bb{G}(p;V)$, in which case we get a $\bb{G}(n-p;N-p)$-bundle on $\bb{G}(p;V)$. These two perspectives will be important in Section~\ref{sec:hodge-loccoh}.

\section{Hodge filtrations on the simple modules $D_p$}\label{sec:Hodge-Dp}

The goal of this section is to characterize the possible Hodge filtrations on a Hodge module whose underlying $\D$-module is $D_p$. We recall the $\GL$-structure of $D_p$ given in (\ref{eq:char-Dp}), and single out the weight 
\begin{equation}\label{eq:def-delta-p}
\delta^p = ((p-n)^p,(p-m)^{n-p})\in W^p,
\end{equation}
noting that $\delta^p(p) = ((p-n)^m)$. Given a $\GL$-subrepresentation $N\subseteq D_p$, we define 
\[\mf{W}(N) = \{ \ll\in\bb{Z}^n_{\dom} : \bb{S}_{\ll(p)}V_1 \oo \bb{S}_{\ll}V_2 \subseteq N\},\]
and note that $\mf{W}(N)$ completely identifies $N$. We also recall the (co)dimension of $\ZA_p$ in $\XA$ from (\ref{eq:dp-cp}). It follows from (\ref{eq:def-Wp}) that if $\ll\in W^p$ then
\[ \ll_{p+1}+\cdots+\ll_n \leq (n-p)\cdot (p-m) = -c_p.\]
To state the main result of this section, we consider the partitioning of $W^p$ as
\begin{equation}\label{eq:partition-Wp}
 W^p = \bigsqcup_{d\geq 0} W^p_d,\mbox{ where }W^p_d = \{\ll\in W^p : \ll_{p+1}+\cdots+\ll_n = -d - c_p\}.
\end{equation}
Using the natural partial order on $\bb{Z}^n$ ($\a\geq\b$ if and only if $\a_i\geq\b_i$ for all $i$), we observe that $W^p_d$ contains finitely many minimal elements with respect to this order, indexed by partitions $\mu\in\mc{P}_{n-p}$ of size $|\mu|=d$. More precisely, these minimal elements are (using the notation in (\ref{eq:lam-dual}))
\begin{equation}\label{eq:def-ll-p-mu}
 \ll^{p,\mu} = \delta^p + \mu^{\vee} = ((p-n)^p,p-m-\mu_{n-p},\cdots,p-m-\mu_1).
\end{equation}

\begin{theorem}\label{thm:Hodge-Dp}
 Suppose that $\mc{M}$ is a Hodge module with underlying $\D$-module $D_p$, and write $F_{\bullet}$ for the Hodge filtration on $D_p$, and $\gr_{\bullet}$ for the associated graded module with respect to $F_{\bullet}$.
 
 (a) There exists (a unique) $l_0\in\bb{Z}$ such that $\mf{W}(\gr_{l_0}(D_p))$ contains $\delta^p$.
 
 (b) We have $F_l(D_p)=0$ for $l<l_0$ and $\mf{W}(\gr_l(D_p)) = W_{l-l_0}^p$ for $l\geq l_0$.
 
 (c) $\mc{M}$ is a pure Hodge module of weight $mn+c_p-2l_0$.
\end{theorem}

\begin{proof}
 By $\GL$-equivariance, the filtered pieces $F_l(D_p)$ are $\GL$-representations. Since $F_{\bullet}$ is a good filtration, it is in particular exhaustive, and therefore $M=\gr_{\bullet}(D_p)$ is isomorphic to $D_p$ as a $\GL$-representation. Since it is a multiplicity free representation and $\delta^p\in W^p$, it follows that there exists a unique index $l_0$ such that $\delta^p\in\mf{W}(\gr_{l_0}(D_p))$, proving (a).
 
For (b), we let $A=S\oo S' \simeq S[y_{i,j}]$ denote the coordinate ring of the cotangent bundle $T^*\XA$, as in Section~\ref{subsec:matrices}. It is a graded $S$-algebra with $S$ placed in degree $0$, and $\deg(y_{i,j})=1$. The associated graded $M=\gr_{\bullet}(D_p)$ is a graded $A$-module, with $M_l=\gr_l(D_p)$ for all $l\in\bb{Z}$. By \cite[Remark~1.5]{raicu-dmods}, the support of $M$ (which is the characteristic variety of $D_p$) is irreducible (equal to the conormal variety $C_p$). By \cite[Lemme~5.1.13]{saito-PRIMS}, $M$ is a Cohen--Macaulay module, which implies that the set-theoretic support of any nonzero $m\in M$ is precisely equal to $C_p$. Using (\ref{eq:proj-ICp}), it follows that
 \begin{equation}\label{eq:ann-A-m}
 \op{Ann}_A(m) \cap S \subseteq J_{p+1}\quad\mbox{ and }\quad \op{Ann}_A(m) \cap S' \subseteq J'_{n-p+1}.
 \end{equation}

 Fix a non-zero element $m_0\in\bb{S}_{\delta^p(p)}V_1 \oo \bb{S}_{\delta^p}V_2 \subseteq M_{l_0}$, consider a partition $\mu\in\P_{n-p}$ with $|\mu|=d$ for some $d\geq 0$, and choose any non-zero element $f'_{\mu}\in\bb{S}_{\mu}V_1^{\vee}\oo\bb{S}_{\mu}V_2^{\vee}\subseteq S'$ (where the inclusion comes from the decomposition of $S'$ analogous to (\ref{eq:decomp-S})). The analogue of (\ref{eq:Slam-in-Jp}) for $S'$ implies that $f'_{\mu}\not\in J'_{n-p+1}$, and using (\ref{eq:ann-A-m}) we get that the element $m_{\mu}:=f'_{\mu}\cdot m_0$ is non-zero. Moreover, since $f'_{\mu}\in A_d$, we have that $m_{\mu}\in M_{l_0+d}$. Since $\bb{S}_{\delta^p(p)}V_1$ is one-dimensional, we have using (\ref{eq:lam-dual}) that
 \[ \bb{S}_{\mu}V_1^{\vee} \oo \bb{S}_{\delta^p(p)}V_1 = \bb{S}_{\delta^p(p) + \mu^{\vee}}V_1 = \bb{S}_{\ll^{p,\mu}(p)}V_1,\]
and therefore $m_{\mu} \in \bb{S}_{\ll^{p,\mu}(p)}V_1 \oo \bb{S}_{\ll^{p,\mu}}V_2$, showing that $\ll^{p,\mu}\in\mf{W}(M_{l_0+d})$. Writing $l=l_0+d$, we conclude that all the minimal elements of $W^p_{l-l_0}$ belong to $\mf{W}(M_l)=\mf{W}(\gr_l(D_p))$. Since $M$ and $D_p$ are isomorphic as $\GL$-representations, and since the sets $W^p_d$ partition $W^p=\mf{W}(M)$, it suffices to verify the inclusions $W^p_{l-l_0} \subseteq\mf{W}(M_l)$ for all $l\geq l_0$ in order to conclude (b). To that end, we prove by induction on $d\geq 0$ that $W^p_d \subseteq \mf{W}(M_{l_0+d})$. 

Consider first the case $d=0$ and let $\ll\in W^p_0$, so that $\ll_{p+1}=\cdots=\ll_n=p-m$. We can write $\ll = \delta^p + \gamma$, where $\gamma\in\mc{P}_p$. We choose any non-zero element $f_{\gamma}\in\bb{S}_{\gamma}V_1\oo\bb{S}_{\gamma}V_2 \subset S$, and note that $f_{\gamma}\not\in J_{p+1}$ by (\ref{eq:Slam-in-Jp}). Using (\ref{eq:ann-A-m}), we get that the element $m^{\gamma}_0:=f_{\gamma}\cdot m_0$ is non-zero, and belongs to $M_{l_0}$ since $\deg(f_{\gamma})=0$. As before we have
\[ \bb{S}_{\gamma}V_1 \oo \bb{S}_{\delta^p(p)}V_1 = \bb{S}_{\delta^p(p) + \gamma}V_1 = \bb{S}_{\ll(p)}V_1,\]
hence $m^{\gamma}_0 \in \bb{S}_{\ll(p)}V_1 \oo \bb{S}_{\ll}V_2$, proving that $\ll\in\mf{W}(M_{l_0})$ and concluding the base case of the induction.
 
For the inductive step, suppose that $d>0$ and let $\ll\in W^p_d$. We can write $\ll = \ll^{p,\mu} + \gamma$ for some $\mu\in\mc{P}_{n-p}$, $|\mu|=d$, and $\gamma\in\mc{P}_p$. We choose $f_{\gamma}$ as in the previous paragraph, and consider the element $m_{\mu}^{\gamma} := f_{\gamma} \cdot m_{\mu}\neq 0$ in $M_{l_0+d}$. By the Littlewood--Richardson rule, we have that
\[ \bb{S}_{\gamma}V_2 \oo \bb{S}_{\ll^{p,\mu}}V_2 = \bb{S}_{\ll}V_2 \oplus L,\]
where the representation $L$ is a direct sum of copies of $\bb{S}_{\b}V_2$ with $\b\geq\ll^{p,\mu}$ and
\[ \b_{p+1}+\cdots+\b_n > -d - c_p.\]
It follows that for any such $\b$ we either have $\b\not\in W^p$, or $\b\in W^p_{d'}$ for some $d'<d$. By induction, we know that $W^p_{d'}\subseteq\mf{W}(M_{l_0+d'})$, forcing $m^{\gamma}_{\mu}$ to be entirely contained in the component $\bb{S}_{\ll(p)}V_1 \oo \bb{S}_{\ll}V_2$ of $M$. This shows that $\ll\in\mf{W}(M_{l_0+d})$, concluding the induction step.
  
 To prove (c), we note that by the discussion in Section~\ref{subsec:Hodge-IC} we have $\mc{M}=IC_{\ZA_p}^H(k_0)$ for some $k_0\in\bb{Z}$. Combining the conclusion of (b) with Lemma~\ref{lem:hodgeF-LZX} we get that the Hodge filtration starts in level $l_0 = c_p+k_0$. Moreover, $\mc{M}$ is pure of weight $d_p-2k_0 = d_p+2c_p-2l_0=mn+c_p-2l_0$, as desired.
\end{proof}

\section{The weight filtration and Hodge ideals for the determinant hypersurface}\label{sec:hodge-det}

In this section $\XA\simeq\bb{C}^{n\times n}$ and $\ZA=\ZA_{n-1}\subset \XA$ is the determinant hypersurface. We consider $\mc{O}_{\XA}(*\ZA)\simeq S_{\det}$ as a mixed Hodge module, with a Hodge filtration $F_{\bullet}$ and a weight filtration $W_{\bullet}$. We write $\gr^F_{\bullet}$ and $\gr^W_{\bullet}$ for the corresponding associated graded modules. The main result of this section is the following.

\begin{theorem}
 We have that $\gr^W_w\mc{O}_{\XA}(*\ZA)=0$ if $w<n^2$ or $w>n^2+n$, and
 \[ \gr^W_{n^2+n-p}\mc{O}_{\XA}(*\ZA) = IC_{\ZA_p}^H\left(-{n-p+1\choose 2}\right)\mbox{ for }p=0,\cdots,n.\]
\end{theorem}

Combined with Theorem~\ref{thm:Hodge-Dp}, this result determines the Hodge filtration on $\mc{O}_{\XA}(*\ZA)$, and with that the Hodge ideals $I_k(\ZA)$. We explain the details in Section~\ref{subsec:Hodge-detl}. 

\subsection{The weight filtration on $\mc{O}_{\XA}(*\ZA)$}\label{subsec:weight-detl}

The goal of this section is to explain the proof of Theorem~\ref{thm:weight-filtration}. We write $S_{\det}$ or $D_p$ when we refer to $\D$-modules, and $\mc{O}_{\XA}(*\ZA)$ or $IC_{\ZA_p}^H(k)$ when we want to keep track of the (mixed) Hodge module structure. 

Since distinct $\D$-module composition factors of $S_{\det}$ have distinct support, it follows from the decomposition by strict support of pure Hodge modules \cite[Section~8.3.3(p4)]{hottaetal} that $\gr_w^W(S_{\det})$ is a direct sum of simple $\D$-modules for each $w\in\bb{Z}$. Since the filtration (\ref{eq:filtration-Sdet}) completely characterizes the $\D$-submodule structure of $S_{\det}$, it follows that the only subquotients of $S_{\det}$ that are direct sums of simple modules are the successive quotients in the filtration (\ref{eq:filtration-Sdet}), and hence they are simple. It follows that we can find $w_0>w_1>\cdots>w_n$ such that
\[\gr_{w_p}^W(S_{\det}) = D_p\mbox{ for }p=0,\cdots,n,\]
and $\gr_w^W(S_{\det})=0$ if $w\not\in\{w_0,\cdots,w_n\}$. At the level of Hodge modules, we have
\[\gr_{w_p}^W\mc{O}_{\XA}(*\ZA) = IC_{\ZA_p}^H(k_p)\mbox{ for }p=0,\cdots,n,\]
where $w_p = d_p - 2k_p$ by the discussion in Section~\ref{subsec:Hodge-IC}. Since the restriction of $\mc{O}_{\XA}(*\ZA)$ to the dense orbit is $\mc{O}_{O_n}^H$ of weight $n^2$, we obtain $w_n=n^2$ and $k_n=0$. To prove Theorem~\ref{thm:weight-filtration}, it suffices to check that $w_p=n^2+n-p$ for $p=0,\cdots,n-1$, since then it follows that
\begin{equation}\label{eq:formula-kp}
 k_p = \frac{d_p-w_p}{2} = \frac{p\cdot(2n-p) - n^2 - n + p}{2} = -{n-p+1 \choose 2},
\end{equation}
as desired. Moreover, since the weights $w_i$ are strictly decreasing, it is enough to check that $w_p-w_{p+1}\leq 1$ for $p=0,\cdots,n-1$, which we do next.

Since maps in the category of mixed Hodge modules are strict with respect to the Hodge filtration, it follows that for each $r\in\bb{Z}$, the weight filtration on $S_{\det}$ determines a filtration by $S$-submodules on $F_r(S_{\det})$, with composition factors $F_r(D_p)$ for $p=0,\cdots,n$. As $\GL$-representations, we get a direct sum decomposition
\begin{equation}\label{eq:FrSdet-sum-FrDp}
 F_r(S_{\det}) = F_r(D_0) \oplus F_r(D_1) \oplus \cdots \oplus F_r(D_n).
\end{equation}
We let $l_p$ be the starting level for the Hodge filtration on $D_p$, and note that by Theorem~\ref{thm:Hodge-Dp}(c) we have 
\begin{equation}\label{eq:wp-from-lp}
w_p = n^2 + (n-p)^2 - 2l_p
\end{equation}
It follows from (\ref{eq:FrSdet-sum-FrDp}) and Theorem~\ref{thm:Hodge-Dp} that $\delta^p = ((p-n)^n)\in\mf{W}(F_{l_p}(D_p)) \subseteq \mf{W}(F_{l_p}(S_{\det}))$. Suppose that $0\leq p\leq n-1$. Using (\ref{eq:W-module-interval}), we obtain $\delta^p + (1^{p+1}) \in \mf{W}(F_{l_p}(S_{\det}))$. Notice that
\begin{equation}\label{eq:deltap+1p+1}
 \delta^p + (1^{p+1}) \overset{(\ref{eq:def-ll-p-mu})}{=} \ll^{p+1,(1^{n-p-1})} \overset{(\ref{eq:partition-Wp})}{\in} W^{p+1}_{n-p-1} = \mf{W}(\gr_{l_{p+1}+n-p-1}(D_{p+1})),
\end{equation}
where the last equality follows from Theorem~\ref{thm:Hodge-Dp}(b). Since $\delta^p + (1^{p+1}) \in \mf{W}(F_{l_p}(S_{\det})) \cap W^{p+1}$, it follows from (\ref{eq:FrSdet-sum-FrDp}) that $\delta^p + (1^{p+1}) \in \mf{W}(F_{l_p}(D_{p+1}))$. Combining this with (\ref{eq:deltap+1p+1}), we obtain the inequality
\[ l_{p+1} + n-p-1 \leq l_p,\]
which is equivalent via (\ref{eq:wp-from-lp}) to $w_p\leq w_{p+1}+1$, as desired.

\begin{remark} It was pointed to us by a referee that Theorem~\ref{thm:weight-filtration} can be verified using a microlocal approach based on \cite{gyoja}. Indeed, if we let $\Lambda_p$ denote the conormal variety to the orbit of rank $p$ matrices, then \cite[Example 9.27]{kashi} computes the corresponding microlocal $b$-function for the determinant hypersurface as $b_{\Lambda_p}(s) = (s+1)(s+2)\cdots(s+n-p)$. By \cite{raicu-dmods}, $\Lambda_p$ is the characteristic variety of $D_p$, and each $D_p$ appears with multiplicity one as a $\D$-module composition factor of $\mc{O}_{\XA}(*\ZA)$. It follows from \cite[Equation~4.7(6)]{gyoja} that $D_p$ appears as a composition factor of $\gr^W_{n^2+j}\mc{O}_{\XA}(*\ZA)$ if and only if $b_{\Lambda_p}(s)$ has exactly $j$ integer roots, that is, if and only if $j=n-p$. This shows that $\gr^W_{n^2+n-p}\mc{O}_{\XA}(*\ZA)$ agrees with $IC_{\ZA_p}^H$ up to a Tate twist, which is then computed as in (\ref{eq:formula-kp}) using Theorem~\ref{thm:Hodge-Dp}.
\end{remark}

\subsection{The Hodge filtration on $\mc{O}_{\XA}(*\ZA)$}\label{subsec:Hodge-detl}

In this section we explain the proof of Theorem~\ref{thm:hodge-filtration-detl}. In light of (\ref{eq:FrSdet-sum-FrDp}), in order to prove (\ref{eq:hodge-filtration-detl}) it is enough to check that $\mf{W}(F_k(D_p)) = U_k^p$, which in turn reduces to showing that $\mf{W}(\gr^F_k(D_p)) = U_k^p\setminus U_{k-1}^p$. Notice that
\begin{equation}\label{eq:Ukp-1}
U_k^p\setminus U_{k-1}^p = \left\{\ll\in W^p : \ll_{p+1}+\cdots+\ll_n = -{n-p+1\choose 2} - k\right\} = \begin{cases}
\emptyset & \mbox{if }k<{n-p\choose 2} \\
W^p_d & \mbox{if }d=k-{n-p\choose 2} \geq 0.
\end{cases}
\end{equation}
It follows from (\ref{eq:wp-from-lp}) and the fact that $w_p = n^2+n-p$ that the Hodge filtration on $D_p$ starts in level $l_p = {n-p\choose 2}$. Moreover, using Theorem~\ref{thm:Hodge-Dp}(b) we have that $\mf{W}(\gr^F_k(D_p)) = W^p_{k-l_p}$ for $k\geq l_p$, which by (\ref{eq:Ukp-1}) is equal to $U_k^p\setminus U_{k-1}^p$, as desired.

We now discuss the generation level of the Hodge filtration on $\mc{O}_{\XA}(*\ZA)$. Given a filtered $\D$-module $(M,F_{\bullet})$, we say that the filtration $F_{\bullet}$ is \defi{generated in level $q$} if
$$
F_{\ell}(\D) \cdot F_q(M)=F_{q+\ell}(M) \textnormal{\;\;for all\; $\ell\geq 0$},
$$
where $F_{\bullet}(\D)$ denotes the order filtration on $\D$. The \defi{generation level} of $(M,F_{\bullet})$ is defined to be the minimal $q$ such that $F_{\bullet}$ is generated in level $q$. 

Since $D_0$ is a quotient of $S_{\det}$, and the Hodge filtration on $D_0$ starts in level ${n\choose 2}$, it follows that the generation level for the Hodge filtration on $S_{\det}$ is at least ${n\choose 2}$. To prove the equality, it suffices to check that $\gr_{\bullet}^F(S_{\det})$ is a graded $A$-module generated in degree $\leq{n\choose 2}$. The weight filtration on $S_{\det}$ induces a filtration on $\gr_{\bullet}^F(S_{\det})$ by graded $A$-submodules, with composition factors $\gr_{\bullet}^F(D_p)$, so it suffices to check that the latter are generated in degree $\leq{n\choose 2}$. The proof of Theorem~\ref{thm:Hodge-Dp}(b) shows that $\gr_{\bullet}^F(D_p)$ is generated as an $A$-module by the ($1$-dimensional) isotypic component $\bb{S}_{\delta^p}V_1 \oo \bb{S}_{\delta^p}V_2$ which appears in degree $l_p = {n-p\choose 2} \leq {n\choose 2}$. 

\begin{remark}
 It is not hard to deduce from the preceding arguments that $\gr_{\bullet}^F(D_p)$ is isomorphic to the coordinate ring of the conormal variety $C_p$ (with an appropriate degree shift, and a twist by a $1$-dimensional $\GL$-representation).
\end{remark}

\subsection{Symbolic powers and the Hodge ideals of the determinant hypersurface}\label{subsec:equiv-ideals-filtration}

In this section we prove Theorem~\ref{thm:hodge-ideals-detl}, by showing that (\ref{eq:hodge-filtration-detl}) implies (\ref{eq:Ik=intersec-symbolic}). For $d\geq 0$ we have by \cite[Section~7]{DCEP} that
\[\mf{W}(J_p^{(d)}) = \{ \mu \in \mc{P}_n : \mu_p + \cdots + \mu_n \geq d\}.\]
By adding the redundant term $p=n$ in (\ref{eq:Ik=intersec-symbolic}), we can reformulate the conclusion of Theorem~\ref{thm:hodge-ideals-detl} as
\[ \mu\in \mf{W}(I_k(\ZA)) \Longleftrightarrow \mu_p+\cdots+\mu_n \geq (n-p)\cdot(k-1) - {n-p\choose 2}\mbox{ for }1\leq p\leq n.\]
Since $\mc{O}_{\XA}((k+1)\ZA)$ is the free $S$-module generated by $\det^{-k-1}$, it follows that after letting $\ll = \mu - ((k+1)^n)$ and using (\ref{eq:Fk-to-Ik}), we can rewrite the above equivalence, after manipulations and setting $s=p-1$, as
\begin{equation}\label{eq:equiv-ll-in-Fk}
 \ll\in \mf{W}(F_k(S_{\det})) \Longleftrightarrow \ll_{s+1}+\cdots+\ll_n \geq -{n-s+1\choose 2}-k, \mbox{ for all }0\leq s\leq n-1.
\end{equation}
Since the sets $W^p$ in (\ref{eq:def-Wp}) partition $\bb{Z}^n_{\dom}$ (when $m=n$), it suffices to prove (\ref{eq:equiv-ll-in-Fk}) under the hypothesis that $\ll\in W^p$ for some fixed $p$. We note that the left hand side is then equivalent via (\ref{eq:hodge-filtration-detl}) to the condition $\ll\in U^p_k$. If $p=n$ then $\ll\in\mc{P}_n$ is a partition, and both sides of (\ref{eq:equiv-ll-in-Fk}) are true. We may therefore assume that $0\leq p\leq n-1$. The implication ``$\Longleftarrow$" follows from the definition of the set $U^p_k$ by taking $s=p$.

To prove ``$\Longrightarrow$" we consider any $\ll\in U^p_k$ and note that the inequality in (\ref{eq:equiv-ll-in-Fk}) is satisfied for $s=p$. If $s<p$ then since $\ll_{s+1},\cdots,\ll_{p-1}\geq\ll_p\geq p-n$ we obtain
\[ \ll_{s+1}+\cdots+\ll_n \geq (p-s)\cdot \ll_p + \ll_{p+1}+\cdots+\ll_n \geq (p-s)\cdot(p-n)-{n-p+1\choose 2} - k\geq -{n-s+1\choose 2}-k,\]
so the inequality in (\ref{eq:equiv-ll-in-Fk}) holds for $s<p$. If $s>p$ then we obtain using $p-n\geq\ll_{p+1}\geq\cdots\geq\ll_s$ that
\[ \ll_{s+1}+\cdots+\ll_n = \ll_{p+1}+\cdots+\ll_n - (\ll_{p+1}+\cdots+\ll_s) \geq -{n-p+1\choose 2} - k - (s-p)\cdot(p-n)\geq -{n-s+1\choose 2}-k,\]
so the inequality in (\ref{eq:equiv-ll-in-Fk}) also holds for $s>p$, concluding the proof.

\section{Hodge module structure for local cohomology with support in maximal minors}\label{sec:hodge-loccoh}

We let $\XA = \bb{A}(V_1\oo V_2)$, where $\dim(V_1)=m>n=\dim(V_2)$, and use the notation in Section~\ref{sec:prelim}. As in the Introduction, we write $f_+$ for both the $\D$-module and Hodge module direct image along some map $f$. The goal of this section is to prove the following.

\begin{theorem}\label{thm:Rf+Hodge}
 Let $f:U \lra \XA$ denote the inclusion of the dense orbit $U=O_n$ into $\XA$. We have
 \[ \mc{H}^j(f_+\mc{O}_U^H) = \begin{cases}
 IC_{\ZA_p}^H(k'_p) & \mbox{if }j=(n-p)\cdot(m-n),\ k'_p=-{n-p+1\choose 2}-(n-p)\cdot(m-n), \\
 0 & \mbox{otherwise},
 \end{cases}
 \]
and $\mc{H}^{(n-p)\cdot(m-n)}(f_+\mc{O}_U^H)$ is pure of weight $mn+(n-p)\cdot(m-n+1)$ for $p=0,\cdots,n$.
\end{theorem}

Writing $\ZA = \ZA_{n-1}$ for the complement of $U$ in $\XA$ and using the standard identification (\ref{eq:Rf+=loccoh}) with local cohomology, together with (\ref{eq:Dp-as-loccoh}), it follows that $\mc{H}^j(f_+\mc{O}_U^H)$ is non-zero only for $j=(n-p)\cdot(m-n)$, in which case its underlying $\D$-module is $D_p$. It follows that $\mc{H}^{(n-p)\cdot(m-n)}(f_+\mc{O}_U^H)$ is equal to $IC_{\ZA_p}^H$ up to a Tate twist, and the content of Theorem~\ref{thm:Rf+Hodge} is the determination of the constants $k'_p$. This is equivalent to finding the weights $w'_p$ of the Hodge modules $IC_{\ZA_p}^H(k'_p)$, because of the identity $w'_p = d_p - 2k'_p$ (see Section~\ref{subsec:Hodge-IC} and (\ref{eq:dp-cp})). We note that Theorem~\ref{thm:hodge-filtration-loccoh} is an immediate consequence of Theorem~\ref{thm:Rf+Hodge}, in light of the identification (\ref{eq:Rf+=loccoh}).

As explained in the Introduction, our strategy to prove Theorem~\ref{thm:Rf+Hodge} proceeds as follows: we factor $f$ as in (\ref{eq:factor-f}), and understand $\iota_+\mc{O}_U^H$ using the information for square matrices developed in Section~\ref{sec:hodge-det}; we then study $\pi_+$ using the Decomposition Theorem, as explained in Sections~\ref{subsec:decomp-detl} and~\ref{subsec:decomp-geom-origin}. Based on these preliminaries, the completion of the proof of Theorem~\ref{thm:Rf+Hodge} is explained in Section~\ref{subsec:weights-loccoh}.

\subsection{Decomposition theorem for some standard resolutions of determinantal varieties}\label{subsec:decomp-detl}

For $p\leq n$ we consider the Grassmannian $\bb{G}_p = \bb{G}(p;V_1)$ of $p$-dimensional quotients of $V_1$ (see Section~\ref{subsec:bundles}) and let 
\begin{equation}\label{eq:resn-Yp-of-Zp}
Y_p = \bb{A}_{\bb{G}_p}(\mc{Q}_p\oo V_2) \overset{\pi_p}{\lra} \bb{A}(V_1\oo V_2)=\XA
\end{equation}
denote one of the standard resolutions of singularities of $\ZA_p$ (see for instance \cite[Proposition~6.1.1(a)]{weyman}). The map $\pi_p$ is the composition of the closed immersion $\bb{A}_{\bb{G}_p}(\mc{Q}_p\oo V_2) \hookrightarrow \bb{A}_{\bb{G}_p}(V_1\oo V_2)$ induced by the tautological quotient map $V_1\oo\mc{O}_{\bb{G}_p}\onto\mc{Q}_p$, with the projection $\bb{A}_{\bb{G}_p}(V_1\oo V_2)\simeq\bb{G}_p\times\XA\lra\XA$. Since $\pi_p$ is projective, it follows from the Decomposition Theorem of \cites{BBD,deCat-Mig} that the $\D$-module direct image of $\mc{O}_{Y_p}$ is a complex in the derived category whose cohomology $\mc{H}^j(\pi_{p+}\mc{O}_{Y_p})$ is semisimple for all~$j$. Since the map $\pi_p$ is $\GL$-equivariant, the summands will be simple $\GL$-equivariant $\D$-modules \cite[Section~5.3]{BL}. The relevant multiplicities are computed by the following (see the discussion after Theorem~\ref{thm:decomp-DpY} for the interpretation of the formula below).

\begin{theorem}\label{thm:decomp-detl}
 The $\D$-module direct image of $\mc{O}_{Y_p}$ is given by the formal identity
 \[\sum_{j\in\bb{Z}} \mc{H}^j(\pi_{p+}\mc{O}_{Y_p})\cdot q^j = \sum_{i=0}^p D_i \cdot q^{-(m-n-p+i)\cdot(p-i)}\cdot{m-n\choose p-i}_{q^2}.\]
\end{theorem}

\begin{proof}
 Using the Riemann--Hilbert correspondence, we replace the $\D$-modules in Theorem~\ref{thm:decomp-detl} with the corresponding perverse sheaves: $\mc{O}_{Y_p}$ with $IC_{Y_p}^{\bullet} = \bb{C}_{Y_p}[d_{Y_p}]$, and $D_i$ with $IC_{\ZA_i}^{\bullet}$. We write $f_i(q)\in\bb{Z}[q,q^{-1}]$ for the Laurent polynomials that encode the Decomposition Theorem for $IC_{Y_p}^{\bullet}$:
 \begin{equation}\label{eq:decomp-IC}
  R\pi_{p*}IC_{Y_p}^{\bullet} = \sum_{i=0}^p IC_{\ZA_i}^{\bullet} \cdot f_i(q).
 \end{equation}
 We determine $f_i(q)$ by considering stalk cohomology in (\ref{eq:decomp-IC}), as is done for instance in the proof of \cite[Theorem~6.1]{DT-toric}. We consider any point $x_k\in O_k$ and compute the stalk cohomology on both sides of (\ref{eq:decomp-IC}): since the fiber $\pi_p^{-1}(x_k)$ is isomorphic to the Grassmannian $\bb{G}(p-k;m-k)$, whose Poincar\'e polynomial is ${m-k\choose p-k}_{q^2}$, we get that the cohomology on the left is encoded by the Laurent polynomial
 \[ q^{-d_{Y_p}} \cdot {m-k\choose p-k}_{q^2} = q^{-p\cdot(m+n-p)}\cdot{m-k\choose p-k}_{q^2}.\]
The stalk cohomology of $IC_{\ZA_i,x_k}^{\bullet}$ is computed by a Kazhdan--Lusztig polynomial (see \cite[Theorem~12.2.5]{hottaetal} for the appropriate grading conventions): this is because of the identification of $\XA$ with the opposite dense Schubert cell in the Grassmannian $\bb{G}(n;m+n)$, under which the subvarieties $\ZA_p$ arise as intersections with Schubert varieties indexed by certain Grassmannian permutations \cite[Section~5.2]{Lak-Rag}. The corresponding Kazhdan--Lusztig polynomials are of parabolic type, computed by $q$-binomial coefficients \cite[Lemme~10.1]{las-sch}. The conclusion of this discussion is summarized by the identity
 \[ \sum_{j\in\bb{Z}} h^j(IC_{\ZA_i,x_k}^{\bullet})\cdot q^j = q^{-d_{\ZA_i}} \cdot {n-k\choose i-k}_{q^2}\mbox{ for }i\geq k.\]
 Using the fact that $d_{Y_p}-d_{\ZA_i} = (p-i)\cdot(m+n-p-i)$, it follows from (\ref{eq:decomp-IC}) that for each $k=0,\cdots,p$ we have an identity
 \[ {m-k\choose p-k}_{q^2} = \sum_{i=k}^p f_i(q) \cdot q^{(p-i)\cdot(m+n-p-i)} \cdot {n-k\choose i-k}_{q^2}.\]
 By plugging in $k=p,p-1,\cdots,0$ in this order, the polynomials $f_i(q)$ are uniquely determined by induction. The conclusion of Theorem~\ref{thm:decomp-detl} amounts to showing that $f_i(q)=q^{-(m-n-p+i)\cdot(p-i)}\cdot{m-n\choose p-i}_{q^2}$, which in turn is equivalent to verifying the $q$-binomial identities
\[{m-k\choose p-k}_{q^2} = \sum_{i=k}^p q^{2(p-i)\cdot(n-i)} \cdot {m-n\choose p-i}_{q^2} \cdot {n-k\choose i-k}_{q^2},\mbox{ for all }0\leq k\leq p.\]
Writing $a=m-n$, $b=n-k$, $c=p-k$, $j=p-i$, the above identities become
\begin{equation}\label{eq:qbin-identities}
{a+b\choose c}_q = \sum_{j=0}^c q^{j\cdot(b+c-j)}\cdot{a\choose j}_q \cdot {b\choose c-j}_q.
\end{equation}
We derive (\ref{eq:qbin-identities}) following \cite[Chapter~1]{gasper-rahman}. We have using \cite[Exercise~1.2(vi)]{gasper-rahman} that
 \[ (z;q)_n := (1-z)\cdot(1-zq)\cdots(1-zq^{n-1})=\sum_{j=0}^n {n\choose j}_q\cdot(-z)^q\cdot q^{k\choose 2},\]
We compute the coefficient of $z^c$ on both sides of the identity $(-z;q)_{a+b} = (-zq^b;q)_a\cdot (-z;q)_b$ and deduce
 \[ {a+b\choose c}_q \cdot q^{c\choose 2} = \sum_{j=0}^c {a\choose j}_q \cdot q^{bj}\cdot q^{j\choose 2} \cdot {b\choose c-j}_q \cdot q^{c-j\choose 2}.\]
 Since $bj+{j\choose 2}+{c-j\choose 2}-{c\choose 2} = j\cdot(b-c+j)$, we obtain (\ref{eq:qbin-identities}) after dividing by $q^{c\choose 2}$.
\end{proof}

\subsection{Decomposition theorem for some $\D$-modules of geometric origin}\label{subsec:decomp-geom-origin}

We take $p=n$ in (\ref{eq:resn-Yp-of-Zp}), and write for simplicity $Y=Y_n$, $\pi=\pi_n$. As in Section~\ref{subsec:bundles}, $Y$ has a rank stratification with $\ZA_p^{Y} = \ZA_{\bb{G}_n}(\mc{Q}_n,V_2)$ and $O_p^Y = \ZA_p^Y\setminus\ZA_{p-1}^Y$. We let $D_p^{Y} = \mc{L}(\ZA_p^{Y},Y)$ denote the corresponding simple $\D$-modules. The goal of this section is to explain how Theorem~\ref{thm:decomp-DpY} follows from Theorem~\ref{thm:decomp-detl}.

\begin{proof}[Proof of Theorem~\ref{thm:decomp-DpY}]
 Let $\bb{F} = \bb{F}(n,p;V_1)$ the partial flag variety parametrizing $2$-step flags of quotients of~$V_1$, of ranks $n$ and $p$ respectively. Consider the commutative diagram
\[
\xymatrix{
T = \bb{A}_{\bb{F}}(\mc{Q}_p\oo V_2) \ar[rr]^g \ar[d]_s & & \bb{A}_{\bb{G}_p}(\mc{Q}_p\oo V_2) = Y_p \ar[d]^{\pi_p} \\
Y = \bb{A}_{\bb{G}_n}(\mc{Q}_n\oo V_2) \ar[rr]_{\pi} & & \bb{A}(V_1\oo V_2)=\XA
}
\]
The map $s$ is a small resolution of singularities of the variety $\ZA_p^Y$, so by the Decomposition Theorem we have
\[ s_+\mc{O}_T = \mc{L}(\ZA_p^Y,Y) = D_p^Y,\]
a complex concentrated in degree zero. The map $g$ is a Grassmannian bundle, with fibers isomorphic to $\bb{G}(n-p;m-p)$. Since the Poincar\'e polynomial of $\bb{G}(n-p;m-p)$ is ${m-p\choose n-p}_{q^2}$, we obtain
\[ \sum_{j\in\bb{Z}}\mc{H}^j(g_+\mc{O}_T) \cdot q^j = \mc{O}_{Y_p} \cdot q^{-(n-p)\cdot(m-n)}\cdot {m-p\choose n-p}_{q^2}.\]
Since $g_+\mc{O}_T = \bigoplus_j\mc{H}^j(g_+\mc{O}_T)[-j]$, it follows that
\[ \pi_+D_p^Y = \pi_+(s_+\mc{O}_T) = \pi_{p+}(g_+\mc{O}_T) = \bigoplus_j \pi_{p+}(\mc{H}^j(g_+\mc{O}_T))[-j],\]
and we conclude using Theorem~\ref{thm:decomp-detl}, and the fact that each $\mc{H}^j(g_+\mc{O}_T)$ is a direct sum of copies of $\mc{O}_{Y_p}$.
\end{proof}

\begin{remark} It is notable that the formulas in Theorem~\ref{thm:decomp-DpY} are reminiscent of those for local cohomology with determinantal support \cites{lor-rai,raicu-weyman}. However, we were not able to establish a direct implication between the two contexts. One reason for this is perhaps the fact that, due to the Decomposition Theorem, the modules involved in Theorem~\ref{thm:decomp-DpY} are always semi-simple (and even better, the direct image complexes are formal). By contrast, local cohomology is described by push-forwards along open immersions, which often result in interesting extensions that are typically hard to control. One common feature in both Theorem~\ref{thm:decomp-DpY} and the work on local cohomology is the presence of $q$-binomial coefficients, which compute the Poincar\'e series of Grassmann varieties. The standard desingularizations of determinantal varieties arise as vector bundles over Grassmann varieties, and the fibers of the desingularization maps are themselves Grassmannians: this is potentially an explanation for the ubiquity of $q$-binomial coefficients in the aforementioned works.
\end{remark}

\subsection{The determination of weights}\label{subsec:weights-loccoh}

The results in the previous sections have immediate analogues at the level of Hodge modules. More precisely, it follows from \cite[Section~8.3.3(m8)--(m11)]{hottaetal} that if $D_p^Y$ underlies a pure Hodge module of weight $w$ then $\mc{H}^j(\pi_+D_p^Y)$ underlies a pure Hodge module of weight $w+j$. We record some basic consequences of Theorem~\ref{thm:decomp-DpY} before proving Theorem~\ref{thm:Rf+Hodge}.

\begin{lemma}\label{lem:Di-in-Rpi-DYp}
 (a) If $D_i$ appears as a summand in $\mc{H}^j(\pi_+D_p^Y)$ then $i\leq p$.
 
 (b) For $i\leq p$, the largest $j$ for which $D_i$ appears as a summand in $\mc{H}^j(\pi_+D_p^Y)$ is 
 \[j=(n-p)\cdot(m-n)+(p-i)\cdot(m-n-p+i).\]

 (c) $D_i$ appears as a summand in $\mc{H}^{(n-i)\cdot(m-n)}(\pi_+D_p^Y)$ if and only if $i=p$.
\end{lemma}

\begin{proof} Conclusions (a) and (b) are direct consequences of the formula in Theorem~\ref{thm:decomp-DpY}. Part (c) follows from (a), (b), and the inequalities
\[(n-p)\cdot(m-n)+(p-i)\cdot(m-n-p+i) < (n-i)\cdot(m-n)\mbox{ for }i<p,\]
which can be rewritten as $(p-i)\cdot(m-n-p+i)<(p-i)\cdot(m-n)$, and follow from $p>i$.
\end{proof}

\begin{proof}[Proof of Theorem~\ref{thm:Rf+Hodge}] Using the exact sequence (with $\ZA=\ZA_{n-1}=\XA\setminus U$)
\[ 0 \lra H^0_{\ZA}(\XA,\mc{O}_{\XA}) \lra \mc{O}_{\XA}^H \lra \mc{H}^0(f_+\mc{O}_U^H) \lra H^1_{\ZA}(\XA,\mc{O}_{\XA}) \lra 0,\]
it follows from (\ref{eq:Dp-as-loccoh}) that we have an isomorphism $\mc{H}^0(f_+\mc{O}_U^H)\simeq \mc{O}_X^H = IC_{\ZA_n}^H$, proving the case $p=n$ ($j=0$) of the theorem. We assume from now on that $j=(n-p)\cdot(m-n)$ for $1\leq p\leq n$, so that $\mc{H}^j(f_+\mc{O}_U^H)$ is a pure Hodge module with underlying $\D$-module $D_p$. Our goal is to prove that its weight is $w'_p=mn+(n-p)\cdot(m-n+1)$ (from which it follows that $\mc{H}^j(f_+\mc{O}_U^H)=IC_{\ZA_p}^H(k'_p)$, as desired).

Using the identification $U\simeq O_n^Y \simeq O_n$, with $Y=Y_n$ as in Section~\ref{subsec:decomp-geom-origin}, we can factor $f=\pi\circ\iota$ as in (\ref{eq:factor-f}). Since $Y$ is locally identified with a space of $n\times n$ matrices over a base $\bb{G}_n$ of dimension $n\cdot(m-n)$, it follows from Theorem~\ref{thm:weight-filtration} and the behavior of weights under pull-back that the weight filtration on $\iota_+\mc{O}_U^H=\mc{O}_Y(*\ZA_n^Y)$ satisfies $\gr^W_w\iota_+\mc{O}_U^H=0$ if $w<mn$ or $w>mn+n$, and
\begin{equation}\label{eq:weight-DsY}
\gr_{mn+n-s}^W\iota_+\mc{O}_U^H = IC_{\ZA^Y_s}^H(k_s),\mbox{ with underlying $\D$-module }D_s^Y,
\end{equation}
where $k_s=-{n-s+1\choose 2}$ as in (\ref{eq:formula-kp}). Using the spectral sequence associated to this filtration, it follows that 
\[\mc{H}^j(f_+\mc{O}_U^H) = \mc{H}^j\left(\pi_+\left(\iota_+\mc{O}_U^H\right)\right)\mbox{ is a sub-quotient of }\bigoplus_{s=0}^n \mc{H}^j\left(\pi_+\left(IC_{\ZA^Y_s}^H(k_s)\right)\right).\]
Since $\mc{H}^j(f_+\mc{O}_U^H)$ has underlying $\D$-module $D_p$ and $\mc{H}^j\left(\pi_+\left(IC_{\ZA^Y_s}^H(k_s)\right)\right)$ has underlying $\D$-module $\mc{H}^j(\pi_+D^Y_s)$, it follows from Lemma~\ref{lem:Di-in-Rpi-DYp} that $\mc{H}^j(f_+\mc{O}_U^H)$ is a sub-quotient of $\mc{H}^j\left(\pi_+\left(IC_{\ZA^Y_p}^H(k_p)\right)\right)$ whose weight is by (\ref{eq:weight-DsY})
\[ j + (mn+n-p) = mn + (n-p)\cdot(m-n+1),\]
concluding our proof.
\end{proof}

\section*{Acknowledgements}
 We are grateful to Mircea Musta\c t\u a for advice and many useful conversations regarding this project, and to Mihnea Popa for helping us improve the discussion of the generation level for the Hodge filtration. We thank the anonymous referee for many helpful comments, and for suggesting an alternative approach to Theorem~\ref{thm:weight-filtration}. Experiments with the computer algebra software Macaulay2 \cite{GS} have provided numerous valuable insights. Perlman acknowledges the support of the National Science Foundation Graduate Research Fellowship under grant DGE-1313583. Raicu acknowledges the support of the National Science Foundation Grant DMS-1901886.

\end{document}